\documentclass[11pt]{amsart}
\usepackage{mathrsfs}
\usepackage{mathtools}
\usepackage{amsfonts}
\usepackage{amssymb}
\usepackage{amsxtra}
\usepackage{dsfont}
\usepackage{color}
\usepackage{graphicx}
\usepackage[english,polish]{babel}
\usepackage{enumitem}

\usepackage{newtxmath}
\usepackage{newtxtext}

\usepackage[margin=2.5cm, centering]{geometry}

\usepackage[colorlinks,citecolor=blue,urlcolor=blue,bookmarks=true]{hyperref}

\hypersetup{
pdfpagemode=UseNone,
pdfstartview=FitH,
pdfdisplaydoctitle=true,
pdfborder={0 0 0}, % No link borders
pdftitle={Weak type $(1, 1)$ estimates for maximal functions along $1$-regular sequences of integers},
pdfauthor={Bartosz Trojan},
pdflang=en-US
}

\newcommand{\ZZ}{\mathbb{Z}}
\newcommand{\NN}{\mathbb{N}}

\newcommand{\calC}{\mathcal{C}}
\newcommand{\calL}{\mathcal{L}}

\newcommand{\calO}{\mathcal{O}}

\newcommand{\calB}{\mathcal{B}}

\newcommand{\calA}{\mathcal{A}}

\newcommand{\scrR}{\mathscr{R}}

\newcommand{\scrA}{\mathscr{A}}

\newcommand{\bfB}{\mathbf{B}}

\newcommand{\vphi}{\varphi}

\newcommand{\pl}[1]{\foreignlanguage{polish}{#1}}

\newcommand{\abs}[1]{\lvert {#1} \rvert}

\newcommand{\ind}[1]{{\mathds{1}_{{#1}}}}

\newtheorem*{theorem*}{Theorem}
\newtheorem*{lemma*}{Lemma}

\newcounter{thm}

\theoremstyle{definition}

\title[Weak-type estimates for $1$-regular sequences]
{Weak type $(1, 1)$ estimates for maximal functions along $1$-regular sequences of integers}
\author{Bartosz Trojan}
\address{
	\pl{
	Bartosz Trojan\\
	The Institute of Mathematics\\
	Polish Academy of Sciences\\
	ul. \'Sniadeckich 8\\
	00-656 Warszawa\\
	Poland}
}
\email{btrojan@impan.pl}


\begin{document}
\selectlanguage{english}

\begin{abstract}
	We show the pointwise convergence of the averages
	\[
		\calA_N f(x) = \frac{1}{\# \bfB_N} \sum_{n \in \bfB_N} f(x + n) 
	\]
	for $f \in \ell^1(\ZZ)$ where $\bfB_N = \bfB \cap [1, N]$, and $\bfB$ is a $1$-regular sequence of integers,
	for example $\bfB = \{\lfloor n \log n \rfloor : n \in \NN\}$.
\end{abstract}

\maketitle
Let $(X, \calB, \mu, T)$ be an ergodic dynamical system, that is $(X, \calB, \mu)$ is a $\sigma$-finite measure space
with a measurable and measure preserving transformation $T: X \rightarrow X$. The classical Birkhoff theorem
\cite{birk} says that for any $f \in L^p(X, \mu)$, $p \geq 1$, the limit
\[
	\lim_{N \to \infty} \frac{1}{N} \sum_{n = 1}^N f\big(T^n x\big)
\]
exists for $\mu$-almost all $x \in X$. This classical result motivated others to study ergodic averages of the form
\begin{equation}
	\label{eq:4}
	\frac{1}{N} \sum_{n = 1}^N f\big(T^{a_n} x\big)
\end{equation}
for various unbounded subsets of integers $\{a_n : n \in \NN\}$ and $L^p(X, \mu)$, $p \in [1, \infty)$.
In particular, in his PhD thesis Wierdl, see \cite[Theorem 4.4]{wrl0}, proved that the averages corresponding to
$a_n = \lfloor n (\log n) \rfloor$ converge $\mu$-almost everywhere for functions in $L^p(X, \mu)$, $p \in (1, \infty)$.
On the other hand, Rosenblatt in \cite[Remark 27]{ros}, showed that for every aperiodic dynamical system
$(X, \calB, \mu, T)$ there is a function in $L^p(X, \mu)$, $p \geq 1$, such that the averages along the orbit of the
sequence $a_n = n \lfloor \log n \rfloor$ do not converge on a set of positive measure.
The latter motivated Rosenblatt and Wierdl to formulate \cite[Conjecture 4.1]{RW} which says that for all subsets of
integers $\{a_n : n \in \NN\}$ having zero Banach density, and any aperiodic dynamical system $(X, \calB, \mu, T)$
there is a function in $L^1(X, \mu)$ so that the ergodic averages \eqref{eq:4} do not converge $\mu$-almost
everywhere. 

However, in \cite{bucz1}, Buczolich constructed inductively a sophisticated example disproving the conjecture. 
Nowadays, thanks to Mirek's article \cite{mm1}, a wide class of concrete examples modeled on certain $c$-regular sequences,
$c \in (1, 30/29)$ is known, for which the Rosenblatt--Wierdl conjecture does not hold.
In this paper, we show how to complete the picture
by covering the case $c = 1$, thus answering the question posed in \cite{mm1}. In particular, we obtain that the ergodic
averages for the sequence $a_n = \lfloor n \log n \rfloor$ converge $\mu$-almost everywhere for all functions in 
$L^1(X, \mu)$, which is in a strong contrast with Rosenblatt's observation \cite[Remark 27]{ros}. 

Let us highlight the difference between discrete and continuous averaging operators. In 1980s Bourgain \cite{bou}
established the pointwise convergence of the ergodic averages along the sequence $a_n = n^2$ for functions
in $L^p(X, \mu)$, $p > 1$. At that time it was not clear what one should expect at the endpoint $p = 1$. Their continuous
counterpart has a form
\begin{equation}
	\label{eq:6}
	\frac{1}{t} \int_0^t f(x - y^2) {\: \rm d} y.
\end{equation}
Observe that by a simple change of variables, we get
\begin{align*}
	\bigg| \frac{1}{t} \int_0^t f(x - y^2) {\: \rm d}y\bigg|
	&=
	\bigg| \frac{1}{2 t} \int_0^{t^2} f(x - y) \frac{{\rm d}y}{\sqrt{y}} \bigg| \\
	&\leq
	\sum_{n = 0}^\infty \frac{1}{2t} \int_{2^{-n-1} t^2}^{2^{-n} t^2} |f(x - y)| \frac{{\rm d}}{\sqrt{y}} \\
	&\leq
	\sum_{n = 0}^\infty 2^{-\frac{n}{2}} \cdot
	\sup_{r > 0}{\frac{1}{r} \int_0^r |f(x - y)| {\: \rm d} y}.
\end{align*}
Therefore, by the classical Hardy--Littlewood maximal inequality, the maximal function corresponding to the averages 
\eqref{eq:6} satisfies weak type $(1, 1)$ estimates. It might suggest that the same holds true for the discrete case,
however in 2010 Bucholicz and Mauldin \cite{BM} showed that the maximal function corresponding to the averages
along $a_n = n^2$ is \emph{not} weak type $(1, 1)$. This illustrates that the phenomena occurring in discrete
setting may completely differ from what is known for the continuous analogues, mainly due to their arithmetic nature. 

Before we formulate our result, let us recall the class of $c$-regular functions we are interested in.
Denote by $\mathcal L$ a family of slowly varying functions $\ell: [x_0, \infty) \rightarrow (0, \infty)$  such that
\[
	\ell(x)=\exp\Big(\int_{x_0}^x \frac{\vartheta (t)}{t} {\: \rm d} t \Big)
\]
where $\vartheta\in \calC^2([x_0, \infty))$ is a real function satisfying
\[
	\lim_{x\to\infty}
	\vartheta(x)=0,\qquad \lim_{x\to\infty}x\vartheta'(x)=0,\qquad
	\lim_{x\to\infty}x^2\vartheta''(x)=0.
\]
In $\calL$ we distinguish a subfamily $\calL_0$ consisting of slowly varying functions
$\ell:[x_0, \infty) \rightarrow (0, \infty)$ such that
\begin{align*}
	\ell(x)=\exp\Big(\int_{x_0}^x\frac{\vartheta (t)}{t} {\: \rm d} t\Big)
\end{align*}
where  $\vartheta \in \calC^2([x_0, \infty))$ is a positive decreasing real function satisfying
\begin{align*}
	\lim_{x\to\infty} \vartheta(x)=0,\qquad
	\lim_{x\to\infty} \frac{x\vartheta'(x)}{\vartheta(x)}=0, \qquad
	\lim_{x\to\infty}\frac{x^2\vartheta''(x)}{\vartheta(x)}=0,
\end{align*}
and for every $\varepsilon>0$ there is a constant $C_{\varepsilon}>0$ such that
$1\le C_{\varepsilon}\vartheta(x)x^{\varepsilon}$ and $\lim_{x\to\infty}\ell(x)=\infty$.
Finally, for every $c \in [1, 30/29)$ let $\mathcal R_c$ be the family of increasing, convex, regularly-varying functions
$h:[x_0, \infty) \rightarrow [1, \infty)$ of the form
\[
	h(x)=x^c L(x)
\]
where $L \in \mathcal L$. If $c=1$ we impose that $L \in \mathcal L_0$.

We are interested in subsets of integers $\bfB$ having a form
\[
	\bfB= \big\{ \lfloor h(m) \rfloor : m \in \NN \big\}
\]
where $h$ is a fixed function belonging to $\scrR_c$. In view of \cite[Lemma 2.14]{mm1}, there is $\delta > 0$ so that
\[
	\# \bfB_N = \vphi(N) \Big(1 + \calO\big(N^{-\delta}\big)\Big) 
\]
where $\bfB_N = \bfB \cap [1, N]$, and $\vphi$ is the inverse function to $h$. In this paper, we study the pointwise
convergence of the ergodic averages
\[
	\scrA_N f(x) = \frac{1}{\# \bfB_N} \sum_{n \in \bfB_N} f\big(T^n x \big).
\]
In view of \cite[Section 4]{mm1}, it is enough to show the weak type $(1, 1)$ maximal estimates for $\scrA_N$. Thanks
to the Calder\'on transference principle \cite{cald} we reduce the problem to the model dynamical
system, that is the integers $\ZZ$ with the counting measure and the shift operator. In this context, for a function
$f \in \ell^1(\ZZ)$ we set
\[
	\calA_N f(x) = \frac{1}{\# \bfB_N}
	\sum_{n \in \bfB_N} f(x + n).
\]
Our aim is to show the following theorem.
\begin{theorem*}
	Let $c = 1$. There is $C > 0$ so that for all $f \in \ell^1(\ZZ)$, 
	\begin{equation}
		\label{eq:5}
		\sup_{\lambda > 0}{
		\lambda \cdot
		\#\Big\{x \in \ZZ : \sup_{N \in \NN}{|\calA_N f(x)|} > \lambda\Big\} }
		\leq
		C \|f\|_{\ell^1}.
	\end{equation}
\end{theorem*}
In view of \cite[Section 5]{mm1}, to obtain \eqref{eq:5} it is sufficient to prove \cite[Lemma 5.1]{mm1},
which is the subject of the following lemma. In fact, its proof is valid not only for $c = 1$ but for the whole range
of the parameter $c \in [1, 30/29)$.
\begin{lemma*}
	There is $C > 0$ such that for all $N \in \NN$, and $x \in \ZZ$ such that $0 < \abs{x} \leq \varphi(N)$,
	\begin{equation}
		\label{eq:3}
		K_N * \check{K}_N(x) \leq \frac{C}{N}
	\end{equation}
	where
	\[
		K_N = \frac{1}{\# \bfB_N} \sum_{N/4 < n \leq N} \delta_n \ind{\bfB}(n)
	\]
	and $\check{K}_N(x) = K_N(-x)$.
\end{lemma*}
\begin{proof}
	Since $K_N * \check{K}_N$ is symmetric, we can restrict attention to $0 < x <  \vphi(N)$. Then
	\begin{align}
		\nonumber
		K_N * \check{K}_N(x) 
		&=
		\frac{1}{(\#\bfB_N)^2} \sum_{N/4 < n, m \leq N} \delta_n * \delta_{-m}(x) \ind{\bfB}(n) \ind{\bfB}(m)\\
		\nonumber
		&=
		\frac{1}{(\#\bfB_N)^2} \sum_{N/4 < n, m \leq N} \delta_{n-m} ( x ) \ind{\bfB}(n) \ind{\bfB}(m)\\
		\label{eq:2}
		&=
		\frac{1}{(\#\bfB_N)^2} \sum_{N/4 < n, n+x \leq N}  \ind{B}(n) \ind{B}(x+n).
	\end{align}
	Hence, our aim is to estimate the cardinality of the set $M(x, N)$ where
	\[
		M(x, N) = \left\{
		\tfrac{1}{4} N < n \leq N : \tfrac{1}{4} N < n+x \leq N \text{ and }  n, n+x \in \bfB
		\right\}.
	\]
	Let us notice that, $n \in \bfB_N \setminus \bfB_{N/4}$ implies that there is $m \in \NN$ so that
	$n = \lfloor h(m) \rfloor$, thus
	\[
		\tfrac{1}{4} N < h(m) \leq N+1,
	\]
	that is
	\[
		\vphi(N/4) < m \leq \vphi(2N).
	\]
	Next, $n + x \in \bfB_N \setminus \bfB_{N/4}$, translates into
	\[
		n + x = \lfloor h(k) \rfloor,
	\]
	for some $k \in \NN$ satisfying
	\[
		\vphi(N/4) < k \leq \vphi(2N).
	\]
	Hence,
	\begin{align*}
		x \leq h(k) - n \leq h(k) - h(m) + 1
	\end{align*}
	and
	\begin{align*}
		x > h(k) - n - 1 \geq h(k) - h(m) - 1.
	\end{align*}
	Therefore, the cardinality of $M(x, N)$ is bounded by the number of pairs $(k, m) \in \NN^2$ so that
	\begin{equation}
		\label{eq:1}
		\left\{
		\begin{aligned}
			\vphi(N/4) \leq m, k \leq \vphi(2N), \\
			x-1 \leq h(k) - h(m) < x+1.
		\end{aligned}
		\right.
	\end{equation}
	To improve the counting in \cite[Lemma 5.1]{mm1}, for a given $m \in \NN$ satisfying
	\[
		\vphi(N/4) \leq m \leq \vphi(2N),
	\]
	we estimate the number of $0 < s \leq \vphi(2N)$ so that
	\[
		x-1 \leq h(m+s) - h(m) \leq x+1.
	\]
	Let $g(s) = h(m+s) - h(m)$. Observe that $g$ is unbounded, increasing and $g(0) = 0$. Therefore, there are
	$s_1 \leq s_2$ such that
	\[
		g(s_1-1) < x-1 \leq g(s_1) \leq g(s_2) \leq x+1 < g(s_2+1).
	\]
	Our aim is to estimate $s_2 - s_1+1$. Let us compute the difference $g(s+1) - g(s)$. By the mean value theorem
	there is $\xi \in [0, 1]$ such that
	\begin{align*}
		g(s+1)-g(s) &= h(m+s+1) - h(m) - h(m+s) + h(m) \\
		&= h(m+s+1) - h(m+s) \\
		&= h'(m+s+\xi).
	\end{align*}
	Hence,
	\footnote{We write $A \simeq B$ if there are absolute constants $C_2 \geq C_1 > 0$ such that
	$C_1 A \leq B \leq C_2 A$.}
	\[
		g(s+1)-g(s) \simeq h'\big(\vphi(N)\big) \simeq \frac{N}{\vphi(N)},
	\]
	where the last estimate follows by \cite[Lemma 2.1]{mm1}. Thus
	\[
		s_2 - s_1 + 1 \simeq \frac{\vphi(N)}{N},
	\]
	and so 
	\[
		\#(M(x, N)) \leq C \vphi(N) \frac{\vphi(N)}{N}.
	\]
	for some $C > 0$. Now, by \eqref{eq:2}, 
	\[
		K_N * \check{K}_N(x) \leq \frac{C}{(\#\bfB_N)^2} \frac{\vphi(N)^2}{N},
	\]
	which together with $\#\bfB_N \simeq \vphi(N)$, leads to \eqref{eq:3}, and the lemma follows.
\end{proof}

\begin{bibliography}{discrete}
        \bibliographystyle{amsplain}
\end{bibliography}

\end{document}